\documentclass[12pt,a4paper]{article}

\usepackage{amsfonts,amsmath, amssymb}
\usepackage{amsthm}
\usepackage{graphicx}
\usepackage{longtable}
\usepackage{float}
\usepackage{array}
\usepackage{enumerate}
\usepackage{threeparttable}
\usepackage{mathrsfs}
\usepackage[text={15cm,24cm},top=2cm,margin=2cm]{geometry} %layout
\usepackage{listings}
\usepackage{url}
\usepackage{subfigure}
\usepackage{multirow}
\usepackage{booktabs}
\usepackage{textcomp,booktabs}
\usepackage[usenames,dvipsnames]{color}
\usepackage{colortbl}
\usepackage{lscape}
\usepackage{bm}
\usepackage[numbers,square,comma,sort&compress]{natbib}%,authoryear%bibtex
\usepackage[colorlinks=true,urlcolor=black,citecolor=blue,linkcolor=black,bookmarks=true]{hyperref} %index
\usepackage{verbatim}
\usepackage{ulem}
\usepackage{rotating}

\definecolor{mygray}{gray}{.9}
\usepackage[labelfont={bf,it},textfont={bf,it}]{caption}
\DeclareGraphicsExtensions{.jpg,.pdf,.gif,.png}

\newtheorem{definition}{Definition}
\newtheorem{proposition}{Proposition}
\newtheorem{theorem}{Theorem}

\newtheorem{remark}{Remark}

\newcommand{\FS}[2]{\displaystyle\frac{#1}{#2}}

\newcommand{\R}{\mathbb{R}}

\numberwithin{equation}{section}
\numberwithin{definition}{section}
\numberwithin{remark}{section}
\numberwithin{theorem}{section}
\numberwithin{proposition}{section}
\numberwithin{lemma}{section}
\numberwithin{remark}{section}
\numberwithin{example}{section}
\numberwithin{figure}{section}
\numberwithin{conjecture}{section}
\numberwithin{table}{section}

\begin{document}

\title{Numerical Methods to Compute Stresses and Displacements from Cellular Forces: Application to the Contraction of Tissue}

\author{Q. Peng, F.J. Vermolen}
\date{January 6, 2020}
\maketitle

\begin{abstract}
	We consider a mathematical model for wound contraction, which is based on solving a momentum balance under the assumptions of isotropy, homogeneity, Hooke's Law, infinitesimal strain theory and point forces exerted by cells. However, point forces, described by Dirac Delta distributions lead to a singular solution, which in many cases may cause trouble to finite element methods due to a low degree of regularity. Hence, we consider several alternatives to address point forces, that is, whether to treat the region covered by the cells that exert forces as part of the computational domain or as 'holes' in the computational domain. The formalisms develop into the immersed boundary approach and the 'hole approach', respectively. Consistency between these approaches is verified in a theoretical setting, but also confirmed computationally. However, the 'hole approach' is much more expensive and complicated for its need of mesh adaptation in the case of migrating cells while it increases the numerical accuracy, which makes it hard to adapt to the multi-cell model. Therefore, for multiple cells, we consider the polygon that is used to approximate the boundary of cells that exert contractile forces. It is found that a low degree of polygons, in particular triangular or square shaped cell boundaries, already give acceptable results in engineering precision, so that it is suitable for the situation with a large amount of cells in the computational domain. 
	
\end{abstract}




\section{Introduction}
\label{intro}
\noindent
Wound healing is a complicated, but crucial biological mechanism. In this manuscript, we consider wound healing after skin injury. Since severe (burn) injuries involve a considerable loss of soft tissue, secondary healing takes place. It involves the formation of a blood clot, in case of a cutaneous wound, the regeneration of collagen (extracellular matrix), and re-vascularisation (which is the re-establishment of a small blood vessel network); see \cite{enoch2008basic} for a biological overview. One of the side effects of secondary healing that follows after a serious skin trauma, is skin contraction. Skin contraction takes place as a result of mechanical, pulling forces that are exerted by the cells (i.e. mainly fibroblasts and myofibroblasts) that are responsible for the regeneration of collagen\citep{cumming2009mathematical}. Contractions can result in a significant, temporary, or even permanent decrease of area or volume of the damaged tissue. Reductions by 5-10 \% of the original wound area have been observed in human skin and in mammalian skin of rodents, even larger reductions have been observed. Such a reduction of skin area or volume leaves residual stresses and strains in the newly repaired skin, as well as in its direct surroundings. This may cause discomfort or even painful sensations to the patient and in extreme cases, contractions may lead to dysfunctionalities of joints. If a contraction is so extreme that the patient develops a disability, then the contraction is referred to as a contracture.

For many of the biological mechanisms that take place during wound healing, mathematical models have been developed. The current manuscript focusses on the formation of a contraction post wounding. Fibroblasts enter the wound site as a result of chemotaxis due to the TGF-beta gradient. Next to the regeneration of collagen, fibroblasts also exert pulling forces to their immediate environment\citep{hinz2006masters}. In some cases, due to being triggered by the high concentration of TGF-beta, fibroblasts differentiate to myofibroblasts, which are known to exert even larger forces than fibroblasts. These larger pulling forces result into the contraction of the tissue around the injury towards the wound centre\citep{darby2014fibroblasts, grinnell1994fibroblasts,li2011fibroblasts}.

In the literature, several attempts to model the contraction phenomenon can be found\citep{boon2016multi,koppenol2017biomedical,murphy2012fibrocontractive, olsen1995mechanochemical,ramtani2004mechanical}. The current manuscript focusses on hybrid models for simulating wound contraction in a small scale, where we consider cells as individual entities. We will consider point forces for modelling the balance of momentum, respectively. The modelling framework will entail Dirac Delta functions (distributions), where these pulse-like forces will lead to singularities of the solution in terms of a lower (local) degree of regularity, even such that the solution no longer falls within the finite-element space in which one looks for the solution. Some of the issues have been treated in \cite{bertoluzza2018local}, \cite{gjerde2018splitting} and \cite{scott1973finite}, regarding well-posedness and finite-element solutions. The treatment of momentum using point forces that we consider in the current paper was developed in \cite{vermolen2015semi},
\cite{boon2016multi} and \cite{koppenol2017biomedical}.

The quest of several alternative methods is motivated by finding ways to improve accuracy, and by the need of efficiency to simulate the mechanical processes occurring in the skin after a serious (burn) trauma. There are different approaches that treat point forces on the boundary of a cell. One may include the region covered by the cell as part of computational domain. This idea develops into the immersed boundary approach. On the contrary, the 'hole approach', is based on excluding the cell from the computational domain and treat the cell forces as a boundary condition. In this paper, we will focus on the balance of momentum where inertia is neglected and where we assume Hooke's Law to be satisfied. Further, we will use the infinitesimal strain approach. To the best of our knowledge, this paper is the first study that assesses the relation between the 'hole approach' and the immerse boundary approach both analytically and computationally.

The paper is structured as follows. In Section \ref{BVP}, we will discuss the singularity problem occurring in the solution of partial differential equations. Section \ref{MathsModels} investigates the 'hole approach' as an alternative to the immersed boundary method, and consistency between these approaches is verified. For a large number of cells in the computational domain, various polygonal approximations of the cell boundary are discussed. In Section \ref{Results}, we compare the immersed boundary approach to the 'hole approach' and show the results from the polygonal cell approach using various polygonal degrees. Finally some conclusions are presented.

%\begin{table}
%\begin{tabular}{lll}
%\hline\noalign{\smallskip}
%\end{tabular}
%\end{table}

\section{Boundary Value Problems with Point Source}\label{BVP}
%According to the cell-based model in \cite{koppenol2017biomedical}, cells are taken as isolated individuals, and the forces released by the (myo)fibroblasts are described as a linear summation of Dirac Delta distributions. However, a
From the definition of the Dirac Delta function, it immediately follows that there is a singularity in the solution to the partial differential equations(PDEs) in some cases. This singularity causes that the solution is irregular and even unbounded if the dimensionality exceeds one. If the PDEs are solved in an infinite domain with Dirac Delta distributions, the solution is known as Green's function. Inspired by this, hereby, we use the Green's function as an intermediate to determine whether there is a singular solution in a given finite domain. In the following contents, we will investigate the solutions in Laplacian equation and elasticity equation respectively.

\begin{theorem}
Given an open bounded domain $\boldsymbol{0}\in\Omega\subset\R^d$, $d>1$, and the boundary value problem below:
\begin{equation}
\label{Eq_lapla_rbc}
(BVP_1)\left\{
\begin{aligned}
-\Delta u&=\delta(\boldsymbol{x}),\quad&\mbox{in $\Omega$,}\\
\FS{\partial u}{\partial\boldsymbol{n}}+\kappa u&=0,\quad&\mbox{on $\partial\Omega$.}
\end{aligned}
\right.
\end{equation}
Then there does not exist a solution $u\in H^1(\Omega)$ such that $u$ can solve $(BVP_1)$.
\end{theorem}    
\noindent
\begin{proof}
Considering Laplacian equation with Dirac Delta function in an infinite region
\begin{equation}
\label{Eq_lapla}
-\Delta u=\delta(\boldsymbol{x}),
\end{equation}
the solution to which is known as the Green's function is 
\begin{equation}
\label{Eq_lapla_green}
\hat{u}(\boldsymbol{x})=\left\{
\begin{aligned}
-\FS{1}{2\pi}\log\|\boldsymbol{x}\|,\quad&\mbox{$d=2$,}\\
\FS{1}{d(d-2)a_d}\cdot\FS{1}{\|\boldsymbol{x}\|^{d-2}},\quad&\mbox{$d\geqslant 3$,}
\end{aligned}
\right.
\end{equation}
where $a_d$ is the total 'surface area' of $(d-1)$-dimensional sphere, i.e. $a_d=2\pi^{(d-1)/2}/\Gamma((d-1)/2)$. Here, $\Gamma(t)=\int_{0}^{\infty}x^{t-1}e^{-x}dx$ is Euler's Gamma function.

Denote $v=u-\hat{u}$ and then $u$ is extracted as $u=v+\hat{u}$. Combining Eq (\ref{Eq_lapla_rbc}) and Eq (\ref{Eq_lapla}), a new boundary value problem is derived:
\begin{equation}
\label{Eq_lapla_homo}
(BVP_1')\left\{
\begin{aligned}
-\Delta v&=0,\quad&\mbox{in $\Omega$,}\\
\FS{\partial v}{\partial\boldsymbol{n}}+\kappa v&=-(\FS{\partial\hat{u}}{\partial\boldsymbol{n}}+\kappa\hat{u}),\quad&\mbox{on $\partial\Omega$.}
\end{aligned}
\right.
\end{equation}
The weak form of $(BVP_1')$ is 
\[ \left\{
\begin{aligned}
&\text{Find $v\in H^1(\Omega)$, such that}\\ 
&\int_{\partial\Omega}\kappa v\phi d\Gamma+\int_{\Omega}v\phi d\Omega=-\int_{\partial\Omega}\kappa\hat{u}+\FS{\partial\hat{u}}{\partial\boldsymbol{n}}d\Gamma,\\
&\text{for all $\boldsymbol{\phi}\in H^1(\Omega)$.}
\end{aligned}	
\right.\]
Note that the solution of $v$ is classic, which is a sufficient condition that $v$ is in $H^1$ space. However, the Green's function is not lying in $H^1$, since $$\int_{\boldsymbol{0}\in\Omega}\|\nabla\hat{u}\|^2d\Omega\rightarrow\infty$$ regardless of the dimensions $d>1$. Since $u = \hat{u} + v$, and $\hat{u} \notin H^1(\Omega)$, it immediately follows that $u \notin H^1(\Omega)$. 
\end{proof}
\begin{remark}
The one-dimensional case of Laplacian equation with boundary conditions does not give unboundedness since the Green's function $$\hat{u}=-\|x\|,$$ is piecewise linear. Hence, the solution is in $H^1(\Omega)$.
\end{remark}

Considering the elasticity equation in one dimension with point source, the equations are expressed as 
\begin{eqnarray}
-\frac{d\sigma}{dx}&=\delta(x),&\quad\mbox{Equation of Equlibirum,}\\
\epsilon&=\frac{du}{dx},&\quad\mbox{Strain-Displacement Relation,}\\
\sigma&=E\epsilon,&\quad\mbox{Constitutive Equation}.
\end{eqnarray}
To simplify the equation with $E=1$ here, the equations above can be combined to Laplacian equation in one dimension:
\begin{equation}
-\frac{d^2u}{dx^2}=\delta(x),
\end{equation}
which contains a solution in $H^1(\Omega)$. For dimensions above one, unfortunately, we have found the Green's function in three dimensions in \cite{weinberger2005lecture}. Therefore, the theorem only states the situation in three dimensions.
\begin{theorem}
Given an open bounded domain $\boldsymbol{0}\in\Omega\subset\R^3$, and the boundary value problem below:
\begin{equation}
\label{Eq_elas_rbc}
(BVP_3)\left\{
\begin{aligned}
-\nabla\cdot\boldsymbol{\sigma}&=\boldsymbol{F}\delta(\boldsymbol{x}),\quad&\mbox{in $\Omega$,}\\
\boldsymbol{\sigma}\cdot\boldsymbol{n}+\kappa\boldsymbol{u}&=\boldsymbol{0},\quad&\mbox{on $\partial\Omega$,}
\end{aligned}
\right.
\end{equation}
where the strain tensor and stress tensor are defined as $$\boldsymbol{\epsilon}=\FS{1}{2}\left[\nabla\boldsymbol{u}+(\nabla\boldsymbol{u})^T\right],$$ and $$\boldsymbol{\sigma}=\FS{E}{1+\nu}\left\lbrace \boldsymbol{\epsilon}+tr(\boldsymbol{\epsilon})\left[ \FS{\nu}{1-2\nu}\right] \boldsymbol{I}\right\rbrace,$$ respectively. Then there does not exist a solution $\boldsymbol{u}\in H^1(\Omega)$ such that $\boldsymbol{u}$ can solve $(BVP_3)$.
\end{theorem}
\noindent
\begin{proof}
From \cite{weinberger2005lecture}, the Green's function in three dimensions is 
$$G_{ij}(\boldsymbol{x})=\FS{1}{16\pi\mu(1-\nu)\|\boldsymbol{x}\|}\left( (3-4\nu)\delta_{ij}+\FS{x_ix_j}{\|\boldsymbol{x}\|^2}\right) ,$$
where $\mu$ and $\nu$ is the second Lam\'e parameter and the Poisson ratio, and  $i,j$ present different coordinates. Further, $\delta_{ij}$ represents the Kronecker Delta function. The displacement vector of each coordinate can be expressed by
\begin{equation}
\label{Eq_elas_green}
\hat{u_i}(\boldsymbol{x})=\sum_{j=1}^{3}G_{ij}(\boldsymbol{x})F_j=\sum_{j=1}^{3}\FS{F_j}{16\pi\mu(1-\nu)\|\boldsymbol{x}\|}\left( (3-4\nu)\delta_{ij}+\FS{x_ix_j}{\|\boldsymbol{x}\|^2}\right).
\end{equation} 
Thus, similarly as before, letting $\boldsymbol{v}=\boldsymbol{u}-\boldsymbol{\hat{u}}$, then the problem becomes 
\begin{equation}
\label{Eq_elas_homo}
(BVP_3')\left\{
\begin{aligned}
-\nabla\cdot\boldsymbol{\sigma}(\boldsymbol{v})&=\boldsymbol{0},\quad&\mbox{in $\Omega$,}\\
\boldsymbol{\sigma}(\boldsymbol{v})\cdot\boldsymbol{n}+\kappa\boldsymbol{v}&=-(\boldsymbol{\sigma}(\boldsymbol{n}\cdot\boldsymbol{\hat{u}})+\kappa\boldsymbol{\hat{u}}),\quad&\mbox{on $\partial\Omega$.}
\end{aligned}
\right.
\end{equation}
Again, $\boldsymbol{v}$ gives classical solution in $\boldsymbol{H}^1(\Omega)$, which implies that we only need to determine whether the Green's function Eq (\ref{Eq_elas_green}) is in $\boldsymbol{H}^1(\Omega)$. Due to the complexity of the expression of the Green's function, it is only necessary to prove part of the integral of $\|\nabla\boldsymbol{\hat{u}}\|^2=\sum_{i,j=1}^{3}\|\FS{\partial\hat{u}_i(\boldsymbol{x})}{\partial x_j}\|^2$ is infinite over the domain $\Omega$ containing the original point. Here, we will calculate the integral of $\|\FS{\partial\hat{u}_x(\boldsymbol{x})}{\partial x}\|^2$ as an example:
\begin{equation*}
\left.
\begin{aligned}
&\int_{\boldsymbol{0}\in\Omega}\|\FS{\partial\hat{u}_x(\boldsymbol{x})}{\partial x}\|^2d\Omega\\
&=\int_{\boldsymbol{0}\in\Omega}\left(-\FS{F_x(3-4\nu)}{16\pi\mu(1-\nu)}\FS{x}{(x^2+y^2+z^2)^{3/2}}+\FS{2F_x}{16\pi\mu(1-\nu)}\FS{x}{(x^2+y^2+z^2)^{3/2}}\right.\\
&-\FS{3F_x}{16\pi\mu(1-\nu)}\FS{x^3}{(x^2+y^2+z^2)^{5/2}}+\FS{yF_y+zF_z}{16\pi\mu(1-\nu)}\FS{1}{(x^2+y^2+z^2)^{3/2}}\\
&\left.-\FS{3(yF_y+zF_z)}{2\cdot16\pi\mu(1-\nu)}\FS{x}{(x^2+y^2+z^2)^{5/2}}\right)^2d\Omega.
\end{aligned}
\right.
\end{equation*}
Then we rewrite the equation with spherical coordinates as $$x=r\sin\phi\cos\theta,\quad y=r\sin\phi\sin\theta,\quad z=r\cos\phi,\quad r=\sqrt{x^2+y^2+z^2}.$$
Therefore, 
\begin{equation*}
\left.
\begin{aligned}
&\int_{\boldsymbol{0}\in\Omega}\|\FS{\partial\hat{u}_x(\boldsymbol{x})}{\partial x}\|^2d\Omega\\
&\propto\int_{\boldsymbol{0}\in\Omega'}r^2\sin\phi\left(\FS{\sin\phi\cos\theta}{r^2}+\FS{\sin^3\phi\cos^3\theta}{r^2}+\FS{\sin\phi\sin\theta+\cos\phi}{r^2}\right.\\
&\left.-\FS{\sin^2\phi\cos\theta\sin\theta+\sin\phi\cos\phi\cos\theta}{r^3} \right)^2d\Omega'\\
&=\int_{\boldsymbol{0}\in\Omega'}\FS{1}{r^2}\sin\phi\left(\sin\phi\cos\theta+\sin^3\phi\cos^3\theta+\sin\phi\sin\theta+\cos\phi\right.\\
&\left.-\FS{\sin^2\phi\cos\theta\sin\theta+\sin\phi\cos\phi\cos\theta}{r} \right)^2d\Omega'.
\end{aligned}
\right.
\end{equation*}
Integrating with respect to $r$ and noting that the inferior of the integral is $0$, then 
\begin{equation}
\label{Eq_elas_pf}
\int_{\boldsymbol{0}\in\Omega'}K_1(\phi,\theta)\FS{1}{r^2}+K_2(\phi,\theta)\FS{1}{r^3}d\Omega'\rightarrow\infty,
\end{equation}
where $K_i(\phi,\theta), i=1,2$ is the expression of $\phi$ and $\theta$. For other derivative parts, they end up with the same situation in Eq (\ref{Eq_elas_pf}), that is, for every part of integral $\int_{\boldsymbol{0}\in\Omega}\|\nabla\boldsymbol{\hat{u}}\|d\Omega$, the integral does not exist. Hence, it can be concluded that the Green's function in isotropic open bounded domain is not in $H^1(\Omega)$, which leads to the consequence that the solution to $(BVP_3)$, expressed by $\boldsymbol{u}=\boldsymbol{v}+\boldsymbol{\hat{u}}$, is not in $\boldsymbol{H}^1(\Omega)$ either.  
\end{proof}
\begin{remark}
Theorems 1 and 2 can also be proved for the case of homogeneous Dirichlet boundary conditions.
\end{remark}

\section{Mathematical Models of Point Forces in Wound Healing}\label{MathsModels}
\subsection{The Immersed boundary method in $\R^2$}
\noindent
The (myo)fibroblasts exert pulling forces on their immediate surroundings in the extracellular matrix. These forces are directed towards the cell centre and they cause local displacements and deformation of the extracellular matrix. The combination of all these forces cause a net contraction of the tissue around the region, where the fibroblasts are actively exerting forces. The fibroblasts, which are responsible for the regeneration of collagen, enter the wound area after serious trauma due to chemotaxis. Since after restoration of the collagen, the fibroblasts die as a result of apoptosis (programmed cell death), the forces that they exert on their environment disappear. If the deformations are relatively large, then residual stresses remain and permanent displacements remain. Therefore, we consider two types of forces: temporary forces ($\boldsymbol{f}_t$) and plastic forces ($\boldsymbol{f}_p$). Here, we will only treat the temporary forces and the way we treat them has been formalized by \cite{vermolen2015semi}, \cite{boon2016multi} and \cite{koppenol2017biomedical}.

For the temporary force of cell $i$, the cell boundary $\Gamma^i$ is divided into line segments in the two-dimensional case. We assume that an inward directed force is exerted at the midpoint of every line segment in the normal direction to the line segment. The total force is a linear combination of every force at every segment. Hence, at time $t$, the total temporary force is expressed by 
\begin{equation}
\label{Eq_ElasTempeq}
\boldsymbol{f}_t(t)=\sum_{i=1}^{T_N(t)}\sum_{j=1}^{N_S^i}P(\boldsymbol{x}_j^i(t))\boldsymbol{n}(\boldsymbol{x}_j^i(t))\delta(\boldsymbol{x}-\boldsymbol{x}_j^i(t))\Delta\Gamma_N^{i,j},
\end{equation}  
where $T_N(t)$ is the number of cells at time $t$, $N_S^i$ is the number of line segments of cell $i$, $P(\boldsymbol{x})$ is the magnitude of the pulling force exerted at point $\boldsymbol{x}$ per length, $\boldsymbol{n}(\boldsymbol{x})$ is the unit inward pointing normal vector (towards the cell centre) at position $\boldsymbol{x}$, $\boldsymbol{x}_j^i(t)$ is the midpoint on line segment $j$ of cell $i$ at time $t$ and $\Delta\Gamma_N^{i,j}$ is the length of line segment $j$.

Theoretically, when $N_S^i\to\infty$, i.e. $\Delta\Gamma_N^{i,j}\to0$, Eq (\ref{Eq_ElasTempeq}) becomes
\begin{equation}
\boldsymbol{f}_t(t)=\sum_{i=1}^{T_N(t)}\int_{\partial\Omega^i}P(\boldsymbol{x}^i(t))\boldsymbol{n}(\boldsymbol{x}^i(t))\delta(\boldsymbol{x}-\boldsymbol{x}^i(t))d\Gamma^{i}.
\end{equation}  
Here, $\boldsymbol{x}^i(t)$ is a point on the cell boundary of cell $i$ at time $t$. 

The equation for conservation of momentum over the computational domain $\Omega$ is given by:
\begin{equation*}
-\nabla\cdot\boldsymbol{\sigma}=\boldsymbol{f}.
\end{equation*}
In the above equation inertia has been neglected. We treat the computational domain as a continuous linear isotropic elastic domain. Therefore, we use Hooke's Law:
\begin{equation}
\label{Eq_elas_sigma}
\boldsymbol{\sigma}=\FS{E}{1+\nu}\left\lbrace \boldsymbol{\epsilon}+tr(\boldsymbol{\epsilon})\left[ \FS{\nu}{1-2\nu}\right] \boldsymbol{I}\right\rbrace,
\end{equation}
where $E$ is the Young's modulus of the domain, $\nu$ is Poisson's ratio and $\boldsymbol{\epsilon}$ is the infinitesimal strain tensor, that is,
\begin{equation}
\label{Eq_elas_eps}
\boldsymbol{\epsilon}=\FS{1}{2}\left[\nabla\boldsymbol{u}+(\nabla\boldsymbol{u})^T\right].
\end{equation}
The above PDE provides a good approximation if the displacements are relatively small. Further, we define the inner product of two second-order $n\times n$ tensors (matrices) $\boldsymbol{A}$ and $\boldsymbol{B}$ as follows:
$$\boldsymbol{A}:\boldsymbol{B}=\sum_{i,j=1}^{n}a_{ij}b_{ij},$$
where $a_{ij}$ and $b_{ij}$ are the entries of $\boldsymbol{A}$ and $\boldsymbol{B}$, respectively.

On the outer boundary $\partial \Omega$, we use the following Robin boundary condition $$\boldsymbol{\sigma}\cdot\boldsymbol{n}+\kappa\boldsymbol{u}=\boldsymbol{0},$$
where $\kappa$ is a positive constant representing a spring force constant between the domain of computation and its far away surroundings, and $\boldsymbol{u}$ denotes the displacement vector. Note that if $\kappa\rightarrow\infty$, then $\boldsymbol{u}\rightarrow\boldsymbol{0}$ which represents a fixed boundary, and $\kappa\rightarrow0$ represents a free boundary in the sense that no external force is exerted on the boundary.

For the case of only one cell $i$ in the computational domain, we need to solve the following boundary value problem:
\begin{equation}
\label{Eq_elas_immersed}
\left\{
\begin{aligned}
-\nabla\cdot\boldsymbol{\sigma}&=\sum_{j=1}^{N_S^i}P(\boldsymbol{x}_j^i(t))\boldsymbol{n}(\boldsymbol{x}_j^i(t))\delta(\boldsymbol{x}-\boldsymbol{x}_j^i(t))\Delta\Gamma^{i,j},\quad&\mbox{in $\Omega$,}\\
\boldsymbol{\sigma}\cdot\boldsymbol{n}+\kappa\boldsymbol{u}&=\boldsymbol{0},\quad&\mbox{on $\partial\Omega$.}
\end{aligned}
\right.
\end{equation}
Let $\boldsymbol{V}(\Omega)$ be a completion of the Hilbert space $\boldsymbol{H}^1(\Omega)$ with smooth functions\citep{scott1973finite}, then the corresponding weak form of Eq (\ref{Eq_elas_immersed}) on $\Omega$ is
\[ (WF_I)\left\{
\begin{aligned}
&\text{Find $\boldsymbol{u}\in \boldsymbol{V}(\Omega)$, such that}\\ &\int_{\partial\Omega}\kappa\boldsymbol{u\phi}d\Gamma+\int_{\Omega}\boldsymbol{\sigma}:\nabla\boldsymbol{\phi}d\Omega\\
&=\int_{\Omega}\sum_{j=1}^{N_S^i}P(\boldsymbol{x}_j^i(t))\boldsymbol{n}(\boldsymbol{x}_j^i(t))\delta(\boldsymbol{x}-\boldsymbol{x}_j^i(t))\Delta\Gamma^{i,j}\boldsymbol{\phi}d\Omega\\
&\rightarrow\int_{\Omega}\int_{\partial \Omega_N^i}P(\boldsymbol{x}^i(t))\boldsymbol{n}(\boldsymbol{x}^i(t))\delta(\boldsymbol{x}-\boldsymbol{x}^i(t))\boldsymbol{\phi}d\Gamma^id\Omega,\quad\mbox{as $N_S^i\rightarrow\infty$}\\
&\text{for all $\boldsymbol{\phi}\in \boldsymbol{V}(\Omega)$.}
\end{aligned}	
\right.\]

\subsection{The 'Hole Approach' in $\R^2$}
\noindent
Since the force is actually applied on a continuous curve, rather than working on the complete computational domain, we remove the region occupied by the cell. It leaves the computational domain with a hole that is occupied by the cell. Then the force on the cell boundary is modelled by a boundary condition on the boundary of the hole (cell). Therewith, we have boundary conditions on the external boundary, as well as a force boundary condition on the boundary of the cell. The boundary value problem we are working on becomes
\begin{equation}
\label{Eq_elas_hole}
\left\{
\begin{aligned}
-\nabla\cdot\boldsymbol{\sigma}&=0,\quad&\mbox{in $\Omega\backslash\Omega_C$,}\\
\boldsymbol{\sigma}\cdot\boldsymbol{n}&=P(\boldsymbol{x})\boldsymbol{n}(\boldsymbol{x}),\quad&\mbox{on $\partial\Omega_C$,}\\
\boldsymbol{\sigma}\cdot\boldsymbol{n}+\kappa\boldsymbol{u}&=\boldsymbol{0},\quad&\mbox{on $\partial\Omega$,}
\end{aligned}
\right.
\end{equation}
where $\boldsymbol{n}(\boldsymbol{x})$ is the unit normal vector pointing out of $\Omega\backslash\Omega_C$, $\Omega$ is the complete computational domain including the cell and extracellular regions, $\Omega_C$ is the region occupied by the cell, and $\partial\Omega_C$ is the boundary of the cell. The corresponding weak form for Eq (\ref{Eq_elas_hole}) is
\[ (WF_{H})\left\{
\begin{aligned}
&\text{Find $\boldsymbol{u}\in \boldsymbol{H}^1(\Omega\backslash\Omega_C)$, such that}\\ &\int_{\partial\Omega}\kappa\boldsymbol{u\phi}d\Gamma+\int_{\Omega\backslash\Omega_C}\boldsymbol{\sigma}:\nabla\boldsymbol{\phi}d\Omega=\int_{\partial\Omega_C}P(\boldsymbol{x})\boldsymbol{n}(\boldsymbol{x})\boldsymbol{\phi}d\Gamma,\\
&\text{for all $\boldsymbol{\phi}\in \boldsymbol{H}^1(\Omega\backslash\Omega_C)$.}
\end{aligned}	
\right.\]
Note that to this problem, it can be proved by combining Lax-Milgram's lemma with Korn's Inequality that a unique solution in $\boldsymbol{H}^1(\Omega)$ exists. In the analysis to come, we assume that the 
cell stays at the same position and keeps the same shape, hence we have $\boldsymbol{x}(t) = \boldsymbol{x}$.
\begin{proposition}
Let $\boldsymbol{u}_H$ and $\boldsymbol{u}_I$, respectively, be solutions to the 'hole approach' (see Equation (\ref{Eq_elas_hole})), and to the immersed boundary approach (see Equation (\ref{Eq_elas_immersed})). Let $\partial \Omega_C$ denote the line over which internal forces are exerted, and let $\partial \Omega$ be the outer boundary of $\Omega$. Then as $\Delta \Gamma \longrightarrow 0$, $$\int_{\partial \Omega} \kappa \boldsymbol{u}_H d \Gamma = \int_{\partial \Omega} \kappa \boldsymbol{u}_I d \Gamma = \int_{\partial \Omega_C} P(\boldsymbol{x}) \boldsymbol{n}(\boldsymbol{x}) d \Gamma.$$
\end{proposition}
\noindent
\begin{proof}
To prove that the above equation holds true, we integrate the PDE of both approaches over the computational domain $\Omega$.

For the immersed boundary approach, we get 
$$-\int_{\Omega}\nabla\cdot\boldsymbol{\sigma}d\Omega=\int_{\Omega}\sum_{j=1}^{N_S^i}P(\boldsymbol{x}_j^i)\boldsymbol{n}(\boldsymbol{x}_j^i)\delta(\boldsymbol{x}-\boldsymbol{x}_j^i)\Delta\Gamma^{i,j}d\Omega,$$
then after applying Gauss Theorem in the LHS and simplifying the RHS, we obtain
$$-\int_{\partial\Omega}\boldsymbol{\sigma}\cdot\boldsymbol{n}d\Gamma=\sum_{j=1}^{N_S^i}P(\boldsymbol{x}^i_j)\boldsymbol{n}(\boldsymbol{x}_j^i)\Delta\Gamma^{i,j}.$$
By substituting the Robin's boundary condition and sending $N_S^i\rightarrow\infty$, i.e. $\Delta\Gamma^{i,j}\rightarrow0$, the equation becomes
\begin{equation}
\label{Eq_immersed_integral}
\int_{\partial\Omega}\kappa\boldsymbol{u_I}d\Gamma=\int_{\partial\Omega_C}P(\boldsymbol{x})\boldsymbol{n}(\boldsymbol{x})d\Gamma.
\end{equation}

Subsequently, we do the same thing for the 'hole approach'. Then, we get
$$-\int_{\Omega}\nabla\cdot\boldsymbol{\sigma}d\Omega=0,$$
and we apply Gauss Theorem:
$$-\int_{\partial\Omega\cup\partial\Omega_C}\boldsymbol{\sigma}\cdot\boldsymbol{n}d\Gamma=0,$$
which implies
$$-\int_{\partial\Omega}\boldsymbol{\sigma}\cdot\boldsymbol{n}d\Gamma-\int_{\partial\Omega_C}\boldsymbol{\sigma}\cdot\boldsymbol{n}d\Gamma=0.$$
Using the boundary conditions, we get
$$\int_{\partial\Omega}\kappa\boldsymbol{u_H}d\Gamma=\int_{\partial\Omega_C}P(\boldsymbol{x})\boldsymbol{n}(\boldsymbol{x})d\Gamma,$$
which is exactly the same as Eq (\ref{Eq_immersed_integral}). Hence we proved that
$$\int_{\partial \Omega} \kappa \boldsymbol{u_H} d \Gamma = \int_{\partial \Omega} \kappa \boldsymbol{u_I} d \Gamma = \int_{\partial \Omega_C} P({\bf x}) {\bf n}(\boldsymbol{x}) d \Gamma.$$ 
\end{proof}

Hence, the two different approaches are consistent in the sense of global conservation of momentum and therefore the results from both approaches should be comparable. The only difference between the two approaches is that the 'hole approach' does not consider the stiffness of the cell, since the cell is treated as a hole in the domain. The immersed boundary method contains the internal stiffness of the cell. Therewith, if the cell stiffness is sent to zero, the two formalisms should deliver the same results. Hereby, we are going to prove this transition mathematically and we will see that numerical computations indeed confirm this behaviour.

Before we state and prove a proposition that asserts the transition, we introduce the following energy norm:
\begin{definition}
Given $\boldsymbol{u}\in H^1(\Omega)$, then the energy norm is defined by 
$$\|\boldsymbol{u}\|_{E(\Omega)}=\left(\int_{\Omega}\boldsymbol{\sigma}(\boldsymbol{u}):\boldsymbol{\epsilon}(\boldsymbol{u})d\Omega+\int_{\partial \Omega}\kappa\boldsymbol{u}^2d\Gamma\right)^{1/2},$$
where $\kappa$ is a positive constant. Note that the energy norm is a proper norm according to the definition of norm in \cite{horn2012matrix}. 
\end{definition}

\begin{proposition}\label{prop_2}
Numerical approximations based on simplicial, continuous finite-element basis functions, to the weak forms of the immersed boundary approach in Equation (\ref{Eq_elas_immersed}) and the 'hole approach' in Equation (\ref{Eq_elas_hole}), yield the same results upon using the following stiffness for the immersed boundary approach 
\begin{equation}
E(\boldsymbol{x})=
\begin{cases}
E,\quad&\mbox{$\boldsymbol{x}\in\Omega\backslash\Omega_C$,}\\
0,\quad&\mbox{$\boldsymbol{x}\in\Omega_C$,}
\end{cases}
\label{Eq_stiffness}
\end{equation}
where $E$ is a constant, $\Omega_C$ is the cell region, $\Omega\backslash\Omega_C$ is the extracellular region and $\Omega_C$ is surrounded by $\Omega$.
\end{proposition}
\noindent
\begin{proof}
Due to the symmetry of the tensor $\boldsymbol{\epsilon}(\boldsymbol{\phi})$, $\forall\boldsymbol{\phi}$, it follows that
$$\int_{\Omega}\boldsymbol{\sigma}(\boldsymbol{u}):\nabla\boldsymbol{\phi}d\Omega=\int_{\Omega}\boldsymbol{\sigma}(\boldsymbol{u}):\boldsymbol{\epsilon}(\boldsymbol{\phi})d\Omega.$$
Hence, rewriting the weak form of the immersed boundary approach taking $N_S^i\rightarrow\infty$, i.e. $\Delta\Gamma^{i,j}\rightarrow0$,  $(WF_I)$ becomes
\[\left\{
\begin{aligned}
&\text{Find $\boldsymbol{u} \in {\bf V}(\Omega)$, such that}\\ &\int_{\partial\Omega}\kappa\boldsymbol{u\phi}d\Gamma+\int_{\Omega}\boldsymbol{\sigma}({\bf u}):\boldsymbol{\epsilon}(\boldsymbol{\phi})d\Omega=\int_{\partial\Omega_C}P(\boldsymbol{x})\boldsymbol{n}(\boldsymbol{x})\boldsymbol{\phi}(\boldsymbol{x})d\Gamma,\\
&\text{for all $\boldsymbol{\phi} \in {\bf V}(\Omega)$.}
\end{aligned}	
\right.\]
%	According to the definition of $\boldsymbol{\sigma}$ and in Equation (\ref{Eq_stiffness}), $\boldsymbol{\sigma}$ is linearly related to $E(\boldsymbol{x})$. Hence we can reshape the expression for the stiffness as follows 
%	\begin{equation}
%	E(\boldsymbol{x})=
%	\begin{cases}
%	E,\quad&\mbox{$\boldsymbol{x}\in\Omega\backslash\Omega_C$,}\\
%	0,\quad&\mbox{$\boldsymbol{x}\in\Omega_C$,}
%	%\alpha E,\quad&\mbox{$\boldsymbol{x}\in\Omega_C$,}
%	\end{cases}
%	\label{Eq_stiffness_linear}
%	\end{equation} 
Substituting Eq (\ref{Eq_stiffness}) into the above weak form, implies that $$\int_{\Omega}\boldsymbol{\sigma}(\boldsymbol{u}):\boldsymbol{\epsilon}(\boldsymbol{\phi})d\Omega=\int_{\Omega\backslash\Omega_C}\boldsymbol{\sigma}(\boldsymbol{u}):\boldsymbol{\epsilon}(\boldsymbol{\phi})d\Omega.$$ Hence, the weak form for the adjusted immersed boundary approach, denoted by $(WF_{I'})$ is given by:
\[(WF_{I'})\left\{
\begin{aligned}
&\text{Find $\boldsymbol{u}\in{\bf V}(\Omega)$, such that}\\ &\int_{\partial\Omega}\kappa\boldsymbol{u\phi}d\Gamma+\int_{\Omega\backslash\Omega_C}\boldsymbol{\sigma}(\boldsymbol{u}):\boldsymbol{\epsilon}(\boldsymbol{\phi})d\Omega
&=\int_{\partial\Omega_C}P(\boldsymbol{x})\boldsymbol{n}(\boldsymbol{x})\boldsymbol{\phi}(\boldsymbol{x})d\Gamma,\\
&\text{for all $\boldsymbol{\phi}\in{\bf V}(\Omega)$.}
\end{aligned}	
\right.\]
Recalling the weak form of the 'hole approach':
\[ (WF_{H})\left\{
\begin{aligned}
&\text{Find $\boldsymbol{u}\in \boldsymbol{H}^1(\Omega\backslash\Omega_C)$, such that}\\ &\int_{\partial\Omega}\kappa\boldsymbol{u\phi}d\Gamma+\int_{\Omega\backslash\Omega_C}\boldsymbol{\sigma}(\boldsymbol{u}):\boldsymbol{\epsilon}(\boldsymbol{\phi})d\Omega=\int_{\partial\Omega_C}P(\boldsymbol{x})\boldsymbol{n}(\boldsymbol{x})\boldsymbol{\phi}d\Gamma,\\
&\text{for all $\boldsymbol{\phi}\in \boldsymbol{H}^1(\Omega\backslash\Omega_C)$.}
\end{aligned}	
\right.\]
We are aware that due to the singularity caused by Dirac Delta distributions in the immersed boundary approach, the solution is no longer in $\boldsymbol{H}^1(\Omega)$. Therefore, following the procedure of discretizing the continuous function space in \cite{bertoluzza2018local}, we approximate the solution by the finite element space $\boldsymbol{V}^h(\Omega)\subset \boldsymbol{H}^1(\Omega)$, such that the solution of $(WF_{I'})$ can be found in this subset  that consists of simplex-based basis functions that are continuous. Subsequently, $(WF_{I'})$ is given by 
\[(WF_{I'}^h)\left\{
\begin{aligned}
&\text{Find $\boldsymbol{u}^h\in\boldsymbol{V}^h(\Omega)$, such that}\\ &\int_{\partial\Omega}\kappa\boldsymbol{u}^h\boldsymbol{\phi}^hd\Gamma+\int_{\Omega\backslash\Omega_C}\boldsymbol{\sigma}(\boldsymbol{u}^h):\boldsymbol{\epsilon}(\boldsymbol{\phi}^h)d\Omega
=\int_{\partial\Omega_C}P(\boldsymbol{x})\boldsymbol{n}(\boldsymbol{x})\boldsymbol{\phi}^h(\boldsymbol{x})d\Gamma,\\
&\text{for all $\boldsymbol{\phi}^h\in\boldsymbol{V}^h(\Omega)$.}
\end{aligned}	
\right.\]
Applying the same discretizing procedure on the weak form of the 'hole approach', we derive the updated weak form as follows:
\[ (WF_H^h)\left\{
\begin{aligned}
&\text{Find $\boldsymbol{u}^h\in \boldsymbol{V}^h(\Omega)$, such that}\\ &\int_{\partial\Omega}\kappa\boldsymbol{u}^h\boldsymbol{\phi}^hd\Gamma+\int_{\Omega\backslash\Omega_C}\boldsymbol{\sigma}(\boldsymbol{u}^h):\boldsymbol{\epsilon}(\boldsymbol{\phi}^h)d\Omega=\int_{\partial\Omega_C}P(\boldsymbol{x})\boldsymbol{n}(\boldsymbol{x})\boldsymbol{\phi}^h(\boldsymbol{x})d\Gamma,\\
&\text{for all $\boldsymbol{\phi}^h\in \boldsymbol{V}^h(\Omega)$.}
\end{aligned}	
\right.\]

Note that the above weak forms are identical. Next we demonstrate that the solutions are necessarily the same
(hence not determined up to a function or a constant).
Since we want to prove the consistency of these two approaches, we rewrite $\boldsymbol{u}^h$ in $(WF_{I'}^h)$ into $\boldsymbol{u}_I^h$ and $\boldsymbol{u}_H^h$ in $(WF_H^h)$. Denoting $\boldsymbol{v}^h=\boldsymbol{u}_I^h-\boldsymbol{u}_H^h$ and subtracting the equations in both weak forms, using linearity the weak form for $\boldsymbol{v}^h$ is   
\[ (WF_v^h)\left\{
\begin{aligned}
&\text{Find $\boldsymbol{v}^h\in\boldsymbol{V}^h(\Omega)$, such that}\\ &\int_{\Omega\backslash\Omega_C}\boldsymbol{\sigma}(\boldsymbol{v}^h):\nabla\boldsymbol{\phi}^hd\Omega+\int_{\partial\Omega}\kappa\boldsymbol{v}^h\boldsymbol{\phi}^hd\Gamma =0\\
%&\Rightarrow \|\int_{\Omega\backslash\Omega_C}\boldsymbol{\sigma}(\boldsymbol{v}):\nabla\boldsymbol{\phi}d\Omega+\int_{\partial\Omega}\kappa\boldsymbol{v\phi}d\Gamma\|=\alpha\|\int_{\Omega_C}\boldsymbol{\sigma}(\boldsymbol{u_1}):\boldsymbol{\epsilon}(\boldsymbol{\phi})d\Omega\|,\\
&\text{for all $\boldsymbol{\phi}^h\in\boldsymbol{V}^h(\Omega)$ and $\alpha\geqslant0$.}
\end{aligned}	
\right.\]
Since $\boldsymbol{\phi}^h$ is a test function, which we can choose freely, such that the provided integrals make sense; we choose $\boldsymbol{\phi}^h=\boldsymbol{v}^h$. The equation in weak form $(WF_v)$ becomes
\begin{equation*}
\begin{aligned}
\int_{\Omega\backslash\Omega_C}\boldsymbol{\sigma}(\boldsymbol{v}^h):\boldsymbol{\epsilon}(\boldsymbol{v}^h)d\Omega+\int_{\partial\Omega}\kappa\|\boldsymbol{v}^h\|^2d\Gamma = \| \boldsymbol{v}^h \|^2_{E(\Omega \setminus \Omega_C)}= 0.
%	&\text{(Cauchy-Schwartz Inequality)}\leqslant\alpha(\int_{\Omega_C}\|\boldsymbol{\sigma}(\boldsymbol{u_1^h})\|^2d\Omega)^{1/2}(\int_{\Omega_C}\|\boldsymbol{\epsilon}(\boldsymbol{v^h})\|^2d\Omega)^{1/2}\\
%	&\text{(Continuity/Boundedness, $\exists K>0$)}\leqslant\alpha K(\int_{\Omega_C}\|\nabla\boldsymbol{u_1^h}\|^2d\Omega)^{1/2}(\int_{\Omega_C}\|\nabla\boldsymbol{v^h}\|^2d\Omega)^{1/2}.
\end{aligned}
\end{equation*}

Since the energy norm is a proper norm, 
it can be concluded that $$\boldsymbol{v}^h=\boldsymbol{0},\quad\mbox{in $\Omega$.}$$ 
Hence, we have proved 
$\boldsymbol{u}_I^h = \boldsymbol{u}_H^h$ in $\Omega$.
\end{proof}

In Proposition \ref{prop_2}, we have proved the convergence between the finite element solutions to the adjusted immersed boundary approach and the 'hole approach'. Next to it, we are going to prove the convergence between the finite element solution to the adjusted immersed boundary approach and the (exact) solution to the 'hole approach'.

\begin{proposition}
Let $\boldsymbol{u}_I$, $\boldsymbol{u}_I^h$, $\boldsymbol{u}_H$, $\boldsymbol{u}_H^h$, respectively, be the (exact) solution to $(WF_{I'})$, the finite element solution to $(WF_{I'}^h)$, the (exact) solution to $(WF_H)$, and the finite element solution to $(WF_H^h)$. Suppose that the finite element error between $\boldsymbol{u}_I^h$ and $\boldsymbol{u}_H^h$ 
satisfies (i.e. the finite element method converges as the element size is sent to zero ($h \rightarrow 0$)):
\begin{equation}
\label{Eq_fem_error}
\|\boldsymbol{u}_H-\boldsymbol{u}_H^h\|_{E(\Omega\backslash\Omega_C)} \rightarrow 0, 
\text{ as } h \rightarrow 0.
\end{equation}
Then, 
$$\|\boldsymbol{u}_H-\boldsymbol{u}_I^h\|_{E(\Omega\backslash\Omega_C)} \longrightarrow 0, 
\text{ as } h \rightarrow 0.$$ Hence, $\boldsymbol{u}_I^h\rightarrow\boldsymbol{u}_H$, as $h\rightarrow0$.
\end{proposition}

\begin{proof}
Since the energy norm is a proper norm, we apply the triangle inequality and obtain 
\begin{equation*}
\left.
\begin{aligned}
\|\boldsymbol{u}_H-\boldsymbol{u}_I^h\|_{E(\Omega\backslash\Omega_C)}&= \|\boldsymbol{u}_H-\boldsymbol{u}_H^h+\boldsymbol{u}_H^h-\boldsymbol{u}_I^h\|_{E(\Omega\backslash\Omega_C)}\\
&\leqslant \|\boldsymbol{u}_H-\boldsymbol{u}_H^h\|_{E(\Omega\backslash\Omega_C)}+\|\boldsymbol{u}_H^h-\boldsymbol{u}_I^h\|_{E(\Omega\backslash\Omega_C)}.
\end{aligned}
\right.	
\end{equation*}
From Proposition \ref{prop_2},
$$\| \boldsymbol{u}_H^h - \boldsymbol{u}_I^h \|_{E(\Omega \setminus \Omega_C)} = 0,$$
and combined with the finite element error stated in Eq (\ref{Eq_fem_error}), we obtain 
$$\|\boldsymbol{u}_H-\boldsymbol{u}_I^h\|_{E(\Omega\backslash\Omega_C)} \rightarrow 0, 
\text{ as } h \rightarrow 0,$$
which confirms the convergence between $\boldsymbol{u}_I^h$ and $\boldsymbol{u}_H$, as $h\rightarrow0$.
\end{proof}

\begin{remark}
For the homogeneous Dirichlet boundary condition, all three propositions can be proved analogously.
\end{remark}

\subsection{Polygonal Cell Approach}
\noindent
If we consider a domain in which many cells are moving and exerting forces, then the aforementioned two approaches will be very expensive from a computational point of view. Therefore, we will simplify the cell boundary to a low-order polygon, such as to a triangle or square. Furthermore, if the cell size is smaller than the mesh size, then we cannot break the cell boundary into finite segments by the mesh for both approaches. Inspired by finite boundary segments which actually build up a polygon, we can simulate the circular cell by different kinds of polygons. 

Eq (\ref{Eq_elas_immersed}) is still used as the basis for the computation of the forces that are exerted by the cells. However, we study the use of just a few boundary segments per cell in such a way that the total force exerted by the cell is the same regardless the order of the polygon.

The cells exert forces on their immediate environment and hence all the points of the computational domain will be displaced. The displacement vector will induce a contraction of the near cell region. This contraction is quantified by the area of the near-cell region. According to \cite{lubliner2008plasticity}, for each nodal point, the new position is $$\boldsymbol{x}(t)=\boldsymbol{X}+\boldsymbol{u}(\boldsymbol{x}(t),t),$$ where $\boldsymbol{X}$ stands for the initial position and $\boldsymbol{x}(t)$ is the position at time $t$. Defining the gradient matrix of displacement $\boldsymbol{J}=\nabla_{\boldsymbol{X}}\boldsymbol{u},$ the matrix notation can be worked out as 
\begin{equation}
d\boldsymbol{x}=\FS{\partial\boldsymbol{x}}{\partial\boldsymbol{X}}d\boldsymbol{X}=(\boldsymbol{I}+\nabla_{\boldsymbol{X}}\boldsymbol{u})d\boldsymbol{X}=(\boldsymbol{I}+\boldsymbol{J})d\boldsymbol{X},
\label{Eq_mechan_displacement}
\end{equation}
where $\FS{\partial\boldsymbol{x}}{\partial\boldsymbol{X}}$ is the Jacobian matrix. The volume can be calculated by:
\begin{equation}
d\boldsymbol{x}=\det(\boldsymbol{I}+\boldsymbol{J})d\boldsymbol{X},
\label{Eq_mechan_area}
\end{equation}
that is, theoretically
\begin{equation}
A_\Omega=\int_{\Omega_0}\det(\boldsymbol{I}+\boldsymbol{J})d\boldsymbol{X},
\label{Eq_mechan_area_J}
\end{equation}
where $\Omega_0$ is the initial domain. 

However, to compute the area in Eq (\ref{Eq_mechan_area_J}) numerically, we need to apply quadratures like Newton-C\^otes quadrature or Gaussian quadrature, which increase the computation expense if we want to track the area at each iteration. Thus, to improve the computational efficiency, another possibility to compute the area of $\Omega$ is based on connecting all the nodal points on the boundary to build up a polygon. Then this polygonal area is an approximation of the deformed area since the displacement of each nodal point is available. To calculate the polygon area, one can use shoelace method derived by \cite{meister1769generalia} in 1769. Suppose we have a polygon with $n$ vertices, then the area is calculated by
\begin{equation}
A_{\Omega} \approx A_{SL}=\FS{1}{2}\|\sum_{i=1}^{n}(x_iy_{i+1}-x_{i+1}y_i)\|,
\label{Eq_mechan_area_SL}
\end{equation}
where $(x_i,y_i)$, $i=1,\dots,n$ is the coordinate of vertex $i$ and $(x_{n+1},y_{n+1})=(x_1,y_1)$. Note that the vertices should be sorted in counter clockwise or clockwise direction.

To have a better insight of how these different computational approaches affect the cell and the near-cell region, we calculate the reduction of the area with respect to the initial area. If we denote the area after deformation by $A_\Omega$ and the original area is $A^0_\Omega$, then the ratio is calculated by
\begin{equation} 
r=\FS{\|A_\Omega-A^0_\Omega\|}{A^0_\Omega}.
\label{Eq_ratio}
\end{equation}

\section{Numerical Results}\label{Results}
\subsection{The Immersed Boundary Approach and The 'Hole Approach'}
\noindent
We use the finite element method to analyse the performance of the immersed boundary approach and 'hole approach'. Since we are interested in the behaviour of the solution in the vicinity of the positions where point forces are exerted, we introduce a subdomain $\Omega_w$ near the locations where the point sources are exerted. This near-by subdomain, as well as the entire computational domain and the circular line where the forces are exerted are shown in Figure \ref{Fig_mesh}. The meshes for the two approaches are the same, except for the use of a 'hole' in the hole-approach. The circular curve where the forces are applied models a cell boundary, with its inner region modelling a myofibroblast that exerts forces on its direct environment. %In each approach, we take the circular shape as the standard, and then compare the polygon simulation results with it to estimate how these simulations perform.
\begin{figure}[htpb]
\subfigure[The immersed boundary approach]{
	\includegraphics[width=0.47\textwidth]{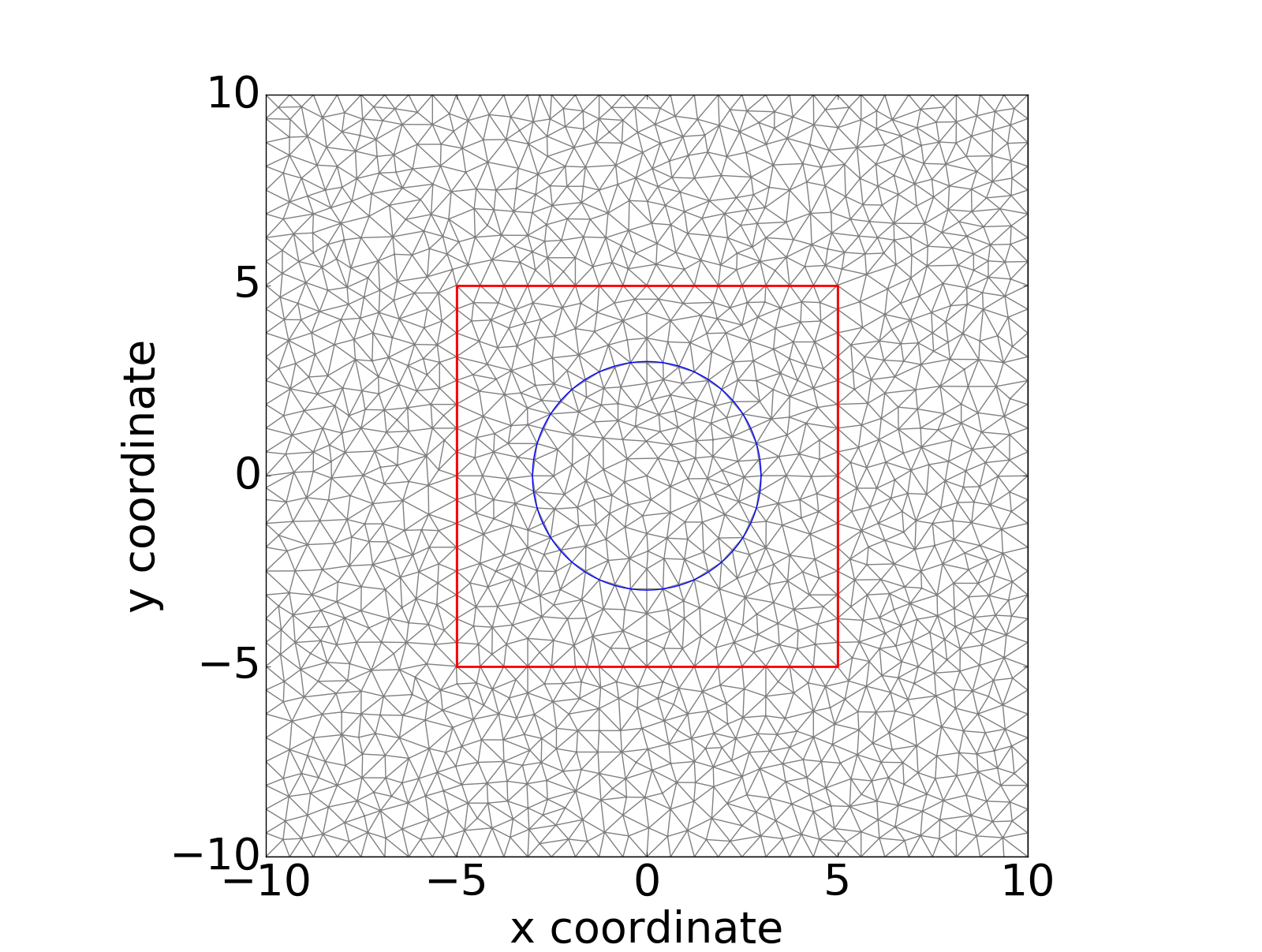}
	\label{fig_delta_mesh}}	
\subfigure[The 'hole approach']{
	\includegraphics[width=0.47\textwidth]{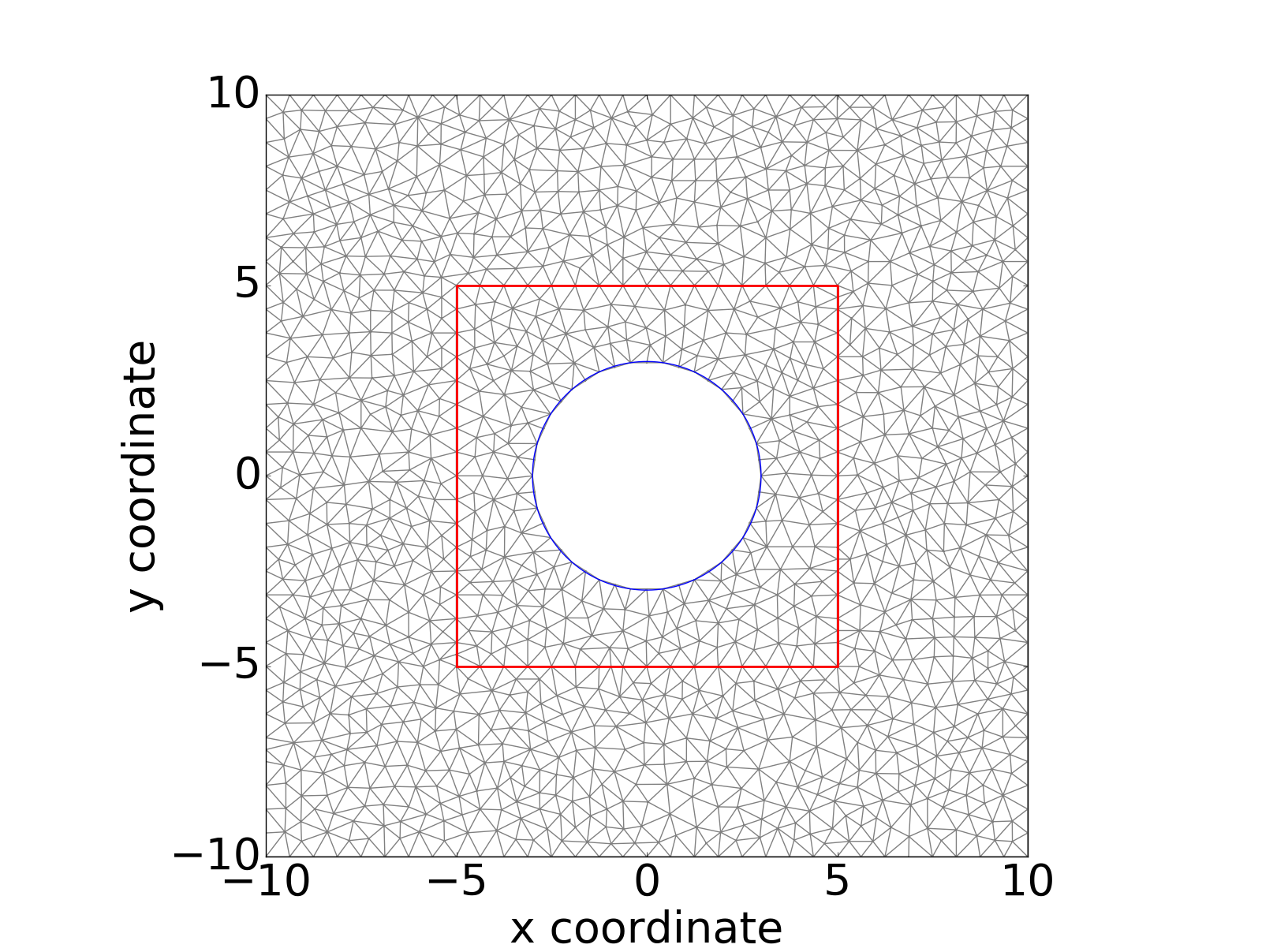}
	\label{fig_hole_mesh}}
\caption{Two subplots show the mesh used for the immersed boundary approach and 'hole approach'. We use ${(-10,10)\times(-10,10)}$ as computational domain, ${(-5,5)\times(-5,5)}$ as near-cell region domain of which the boundary is marked with red lines and the cell is drawn in blue}
\label{Fig_mesh}
\end{figure} 

The values of the parameters used in this simulation have been listed in Table \ref{Tbl_ParaValue}. Note that all these parameter values are only for testing the sensitivity of the approaches.
\begin{table}[htpb]\footnotesize
\centering
\caption{Parameter values}
\begin{tabular}{m{2cm}<{\centering}m{4cm}<{\centering}m{1cm}<{\centering}m{2.5cm}<{\centering}}
	\hline\noalign{\smallskip}
	{\bf Parameter}& {\bf Description} & {\bf Value} & {\bf Dimension} \\
	\noalign{\smallskip}\hline\noalign{\smallskip}
	$E$ & Substrate elasticity & $1$ & $kg/(\mu m\cdot min^2)$\\
	$P$ & Magnitude of the force exerted by the cell & $1$ & $kg\cdot\mu m/min^2$ \\
	$R$ & Cell radius & $3$ & $\mu m$ \\
	$\kappa$ & Boundary condition coefficient & $10$ & $-$ \\
	$\nu$ & Poisson's ratio  & $0.49$ & $-$ \\ 
	\noalign{\smallskip}\hline
\end{tabular}
\label{Tbl_ParaValue}
\end{table}

We compare the results from the immersed boundary approach to the results from the 'hole approach'. Figure \ref{Fig_immersed_hole} displays the initial cell in blue and the nearby region which is included in the red square, as well as its deformations in black curves. It can be seen that there is a large difference between the results from the two approaches. In particular, the magnitude of the displacement from the 'hole approach' is more than $13$ times as large as the displacement from the immersed boundary approach. This discrepancy is caused by the interaction with the region inside the circular cell, which is incorporated in the immersed boundary approach and not in the 'hole approach'. Therefore, we adjust the stiffness of the region inside the circular cell to zero, by Eq (\ref{Eq_stiffness}). However, rather than setting the stiffness modulus to zero inside the cell in implementation, we set the cell stiffness modulus to a small positive constant:
\begin{equation}
E(\boldsymbol{x})=
\begin{cases}
E,\quad&\mbox{$\boldsymbol{x}\in\Omega\backslash\Omega_C$,}\\
\gamma,\quad&\mbox{$\boldsymbol{x}\in\Omega_C$,}
\end{cases}
\label{Eq_ad_stiffness}
\end{equation} 
where $\gamma$ is a small positive constant. In the following contents about the adjusted immersed boundary approach, we use $\gamma=10^{-5}$ if there is no further declaration. Then we redo the simulations and plot the results in Figure \ref{Fig_ad_immersed_hole}. The results of area and total strain energy in the subdomain $\Omega_w$ have been documented in Table \ref{Tbl_ad_immersed_hole}, and as a result of the use of Eq (\ref{Eq_stiffness}), it can be seen that the 'hole approach' and the adjusted immersed boundary approach are consistent since the area reductions are less than a percent. Further, it can be observed that the order of accuracy of the 'hole approach' is slightly better, whereas the adjusted immersed boundary approach is about a factor of four more economical from a computational efficiency point of view. 

\begin{figure}[htpb]
\centering
\subfigure[The immersed boundary approach]{
	\includegraphics[width=0.48\textwidth]{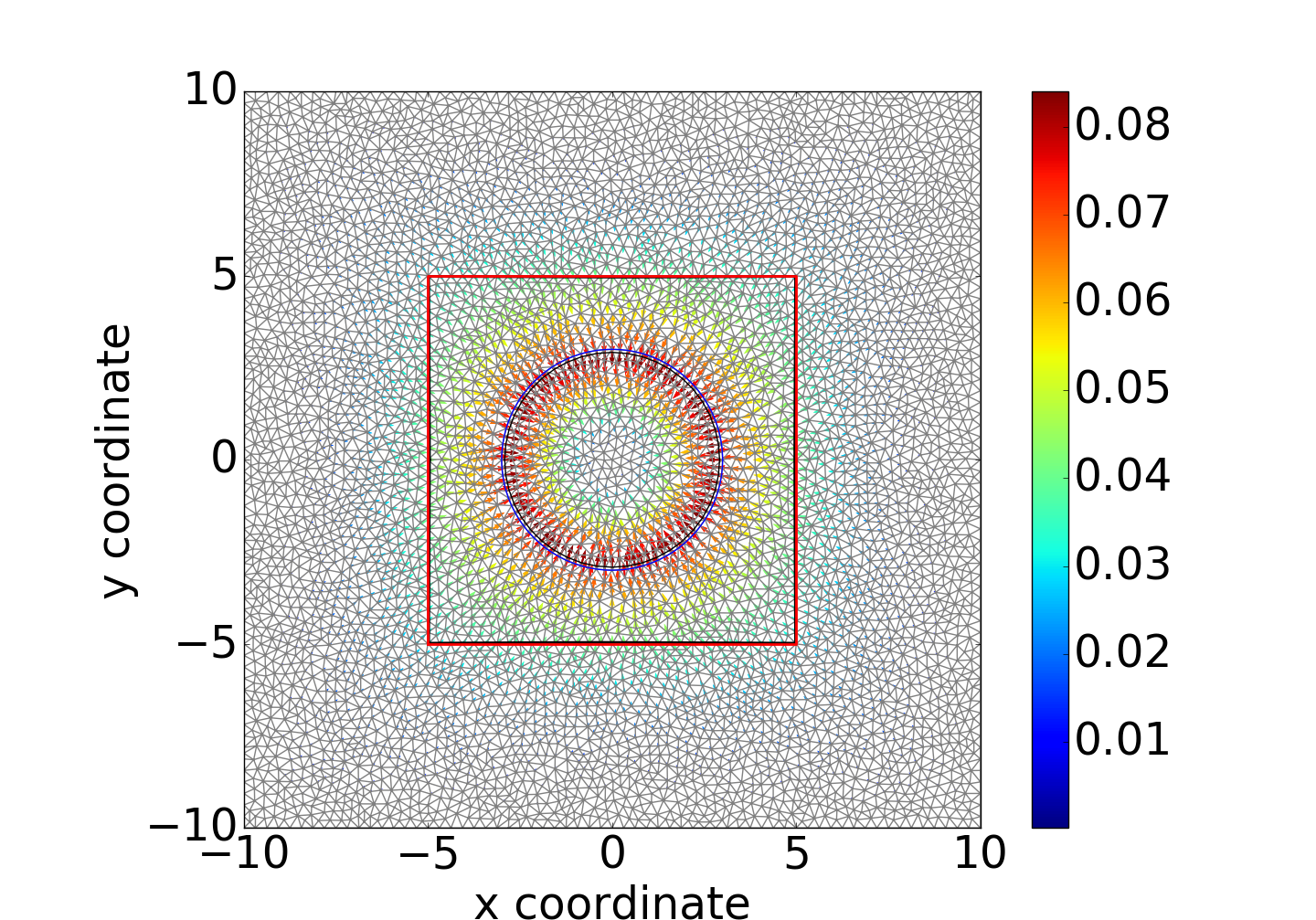}}
\label{fig_immersed}
\subfigure[The 'hole approach']{
	\includegraphics[width=0.48\textwidth]{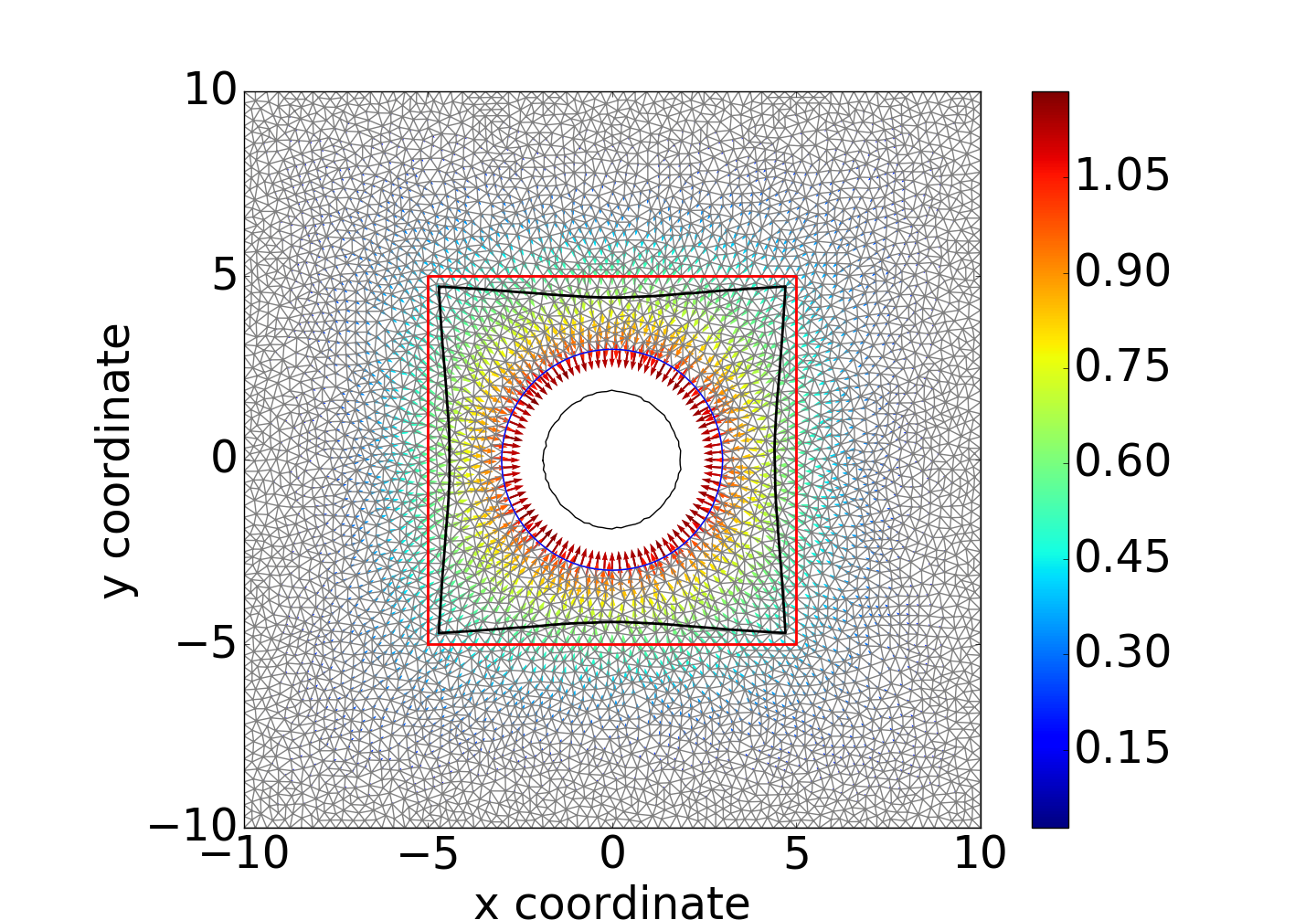}}
\label{fig_hole}
\caption{Displacement results of the immersed boundary approach (Eq (\ref{Eq_elas_immersed})) and the 'hole approach' (Eq (\ref{Eq_elas_hole})) when the same mesh structure used except the hole and the same parameter values applied (Table \ref{Tbl_ParaValue}). The black line shows the deformed cell and $\Omega_w$ and the other colour lines represent the original status}
\label{Fig_immersed_hole}
\end{figure}

\begin{figure}[htpb]
\centering
\subfigure[The adjusted immersed boundary approach]{
	\includegraphics[width=0.48\textwidth]{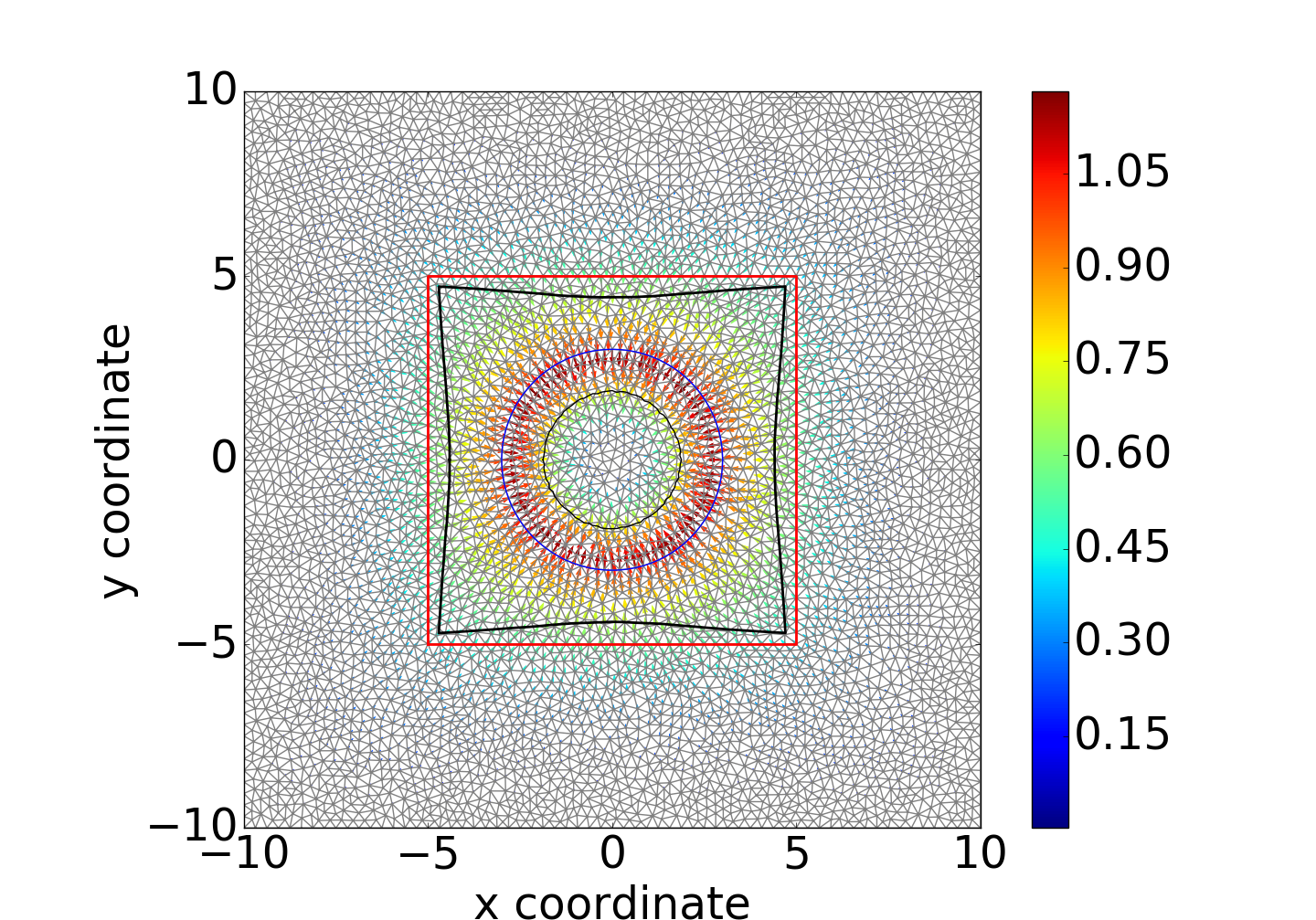}}
\subfigure[The 'hole approach']{
	\includegraphics[width=0.48\textwidth]{hole_cell_r3_60.png}}
\caption{Displacement results of the adjusted immersed boundary approach (Eq (\ref{Eq_elas_immersed}) and Eq(\ref{Eq_ad_stiffness})) and the 'hole approach' (Eq (\ref{Eq_elas_hole})) when the same mesh structure used except the hole and the same parameter values applied (Table \ref{Tbl_ParaValue}). The black line shows the deformed cell and $\Omega_w$ and the other colour lines represent the original status. }
\label{Fig_ad_immersed_hole}
\end{figure}

\begin{table}[htpb]\footnotesize
\centering
\caption{The percentage of area change of cell and vicinity region, and time cost of three approaches}
\begin{tabular}{m{5cm}<{\centering}m{4cm}<{\centering}m{4cm}<{\centering}}
	\hline\noalign{\smallskip}
	& {\bf The immersed boundary approach}& {\bf The 'hole approach'} \\
	\noalign{\smallskip}\hline\noalign{\smallskip}
	Cell Area Reduction Ratio(\%) & $61.92051$ &  $61.92605$ \\
	$\Omega_w$ Area Reduction Ratio(\%) & $17.50153$ & $17.52235$ \\
	Convergence Rate of Strain Energy in $\Omega_w$ &  $1.70656$ & $2.0647$\\
	Time Cost$(s)$ & $1.99139$ & $8.71979$\\
	\noalign{\smallskip}\hline
\end{tabular}
\label{Tbl_ad_immersed_hole}
\end{table}

Due to multiple choices of $\gamma$, the value of $\gamma$ determines the accuracy and convergence of the adjusted immersed boundary approach. In this manuscript, to investigate the effect of $\gamma$, it varies from $10^{-6}$ to $10^{-3}$ with steps of a factor of $10$. In Table \ref{Tbl_ad_immersed_gamma_hole}, besides the area reduction, the convergence rate of the $L_2$-norm of the solution and the total strain energy in $\Omega_w$ are shown. It can be concluded that the value of $\gamma$ does have a modest impact in the current range, and the influences on various categories are distinct. In other words, for the area reduction, it is verified that the smaller value $\gamma$ is, the closer the result is to the one in 'hole approach'. Nevertheless, there is 'bell shape' behaviour appearing for the convergence rate of $\boldsymbol{\|u\|}_{L_2}$, although the differences are not strikingly large. Further, we observed that, in the perspective of the strain energy in $\Omega_w$, the larger $\gamma$ is, the better the convergence rate.  

\begin{table}[htpb]\footnotesize
\centering
\caption{Numerical results of the adjusted immersed boundary approach and the 'hole approach' with multiple choices of $\boldsymbol{\gamma}$}
\begin{tabular}{m{2cm}<{\centering}m{0.5cm}<{\centering}m{3cm}<{\centering}m{3cm}<{\centering}m{3cm}<{\centering}}
	\hline\noalign{\smallskip}
	{\bf Approach} & $\boldsymbol{\gamma}$ & {\bf The Percentage of area reduction(\%)} & {\bf Convergence rate of $\|\boldsymbol{u}\|_{L_2}$} & {\bf Convergence rate of $\int_{\Omega_w}1/2\times\boldsymbol{\sigma}(\boldsymbol{u}):\boldsymbol{\epsilon}(\boldsymbol{u})d\Omega$} \\
	\noalign{\smallskip}\hline\noalign{\smallskip}
	{\bf The 'hole approach'} & $-$ & $17.49741928$ & $1.978019816$ & $2.064701439$  \\
	\midrule
	{\bf The adjusted} & $10^{-3}$ & $17.29570621$ & $1.882445881$ & $1.929776181$ \\
	{\bf immersed} & $10^{-4}$ & $17.48242014$ & $1.984418004$ & $1.704289701$ \\
	{\bf boundary} & $10^{-5}$ & $17.49936018$ & $1.984324634$ & $1.706561293$ \\
	{\bf approach} & $10^{-6}$ & $17.50084960$ & $1.769210872$ & $1.583005166$ \\
	\noalign{\smallskip}\hline
\end{tabular}
\label{Tbl_ad_immersed_gamma_hole}
\end{table}

\subsection{Polygonal Cell Approach}
In the applications that we study, we are interested in multiple cells that are migrating through the computational domain. In typical situations, the cell size is much smaller than the domain size and the cell size could even be smaller than the element size from the discretization. Therefore, it is expensive from a computational point of view to divide the cell boundary into many mesh points and line segments in these applications. Hence, we are interested in the numerical accuracy if each cell is approximated by a simple polygon like a triangle or square instead of a high order polygon. In the presence of multiple small cells, we will study the impact of the polygonal order on the numerical results. The values of the input parameters are given in Table \ref{Tbl_ParaValue_new}. 
\begin{table}[htpb]\footnotesize
\centering
\caption{Parameter values}
\begin{tabular}{m{2cm}<{\centering}m{4cm}<{\centering}m{1cm}<{\centering}m{2.5cm}<{\centering}}
	\hline\noalign{\smallskip}
	{\bf Parameter}& {\bf Description} & {\bf Value} & {\bf Dimension} \\
	\noalign{\smallskip}\hline\noalign{\smallskip}
	$E$ & Substrate elasticity & $1$ & $kg/(\mu m\cdot min^2)$\\
	$P$ & Magnitude of the force exerted by the cell & $10$ & $kg\cdot\mu m/min^2$ \\
	$R$ & Cell radius & $0.1$ & $\mu m$ \\
	$\kappa$ & Boundary condition coefficient & $10$ & $-$ \\
	$\nu$ & Poisson's ratio  & $0.49$ & $-$ \\
	$\lambda$ & Parameter in Point Poisson Process of cells& $15$ & $-$ \\   
	\noalign{\smallskip}\hline
\end{tabular}
\label{Tbl_ParaValue_new}
\end{table}

In the multi-cell simulations, we locate the cells according to a Point Poisson Process with rate parameter $\lambda$, where we choose $\lambda = 15$ from \cite{krieger2013age}. The cell radius has been scaled down to $0.1$ of the radius in the previous calculations. The computational domain and the near-cell region are the same as in the earlier simulations. In order to visualize the deformation of the cell and the subdomain $\Omega_w$, we set the magnitudes of the forces exerted by the cells to $10$. In the simulations, we use the immersed boundary method with low order polygonal approximations of the circular cells. We investigate the performance in terms of the numerical solution with respect to the degree of polygons. An example of a simulation is shown in Figure \ref{Fig_multi_eqarea_sq}, where multiple cells are shown as circles, and the contraction of the region is shown. The cell size is smaller than the mesh size, so we applied the polygonal cell approach here to investigate the area reduction of the region.

The numerical numbers that we investigate are the area reduction due to the pulling forces exerted by the cells and the computation time. In all the calculations where we vary the degree of the polygonal approximation of the cells, we use the same number of cells and the same positions of the centres of the cells. Upon increasing the degree of the polygon, one gradually converges to a circle. In the current computations, we use a maximum number of eight nodes on the cells, that is, we use octagons as the highest polygonal order. The smallest order of polygonal approximation is the triangular shape. We selected the polygons such that the area of each cell is equal in all simulation runs as well as the centres of the cells. 

Figure \ref{Fig_multi_eqarea_ratio_time} displays the computation time and relative reduction of area as a function of polygonal degree with multiple cells. Lower order of polygonal approximation admits the advantage that computation time can be reduced due to a lower number of function evaluations from point forces. In the computations, it has turned out that the use of triangles gave a reduction of computation time of roughly fifty percent with respect to the octagonal representation of the cell boundaries according to the histogram in Figure \ref{Fig_multi_eqarea_ratio_time}. The dash line in Figure \ref{Fig_multi_eqarea_ratio_time} shows that a triangle or square representation of the circles already reproduces the results of the octagonal representation very well, since there is tiny fluctuation. In one word, due to the efficient computation time and good reproduction of the octagonal results in area reduction, we recommend to approximate the cell boundary by a triangle or square if a large number of small cells are used.

\begin{figure}[htpb]
\centering
\includegraphics[width=0.7\textwidth]{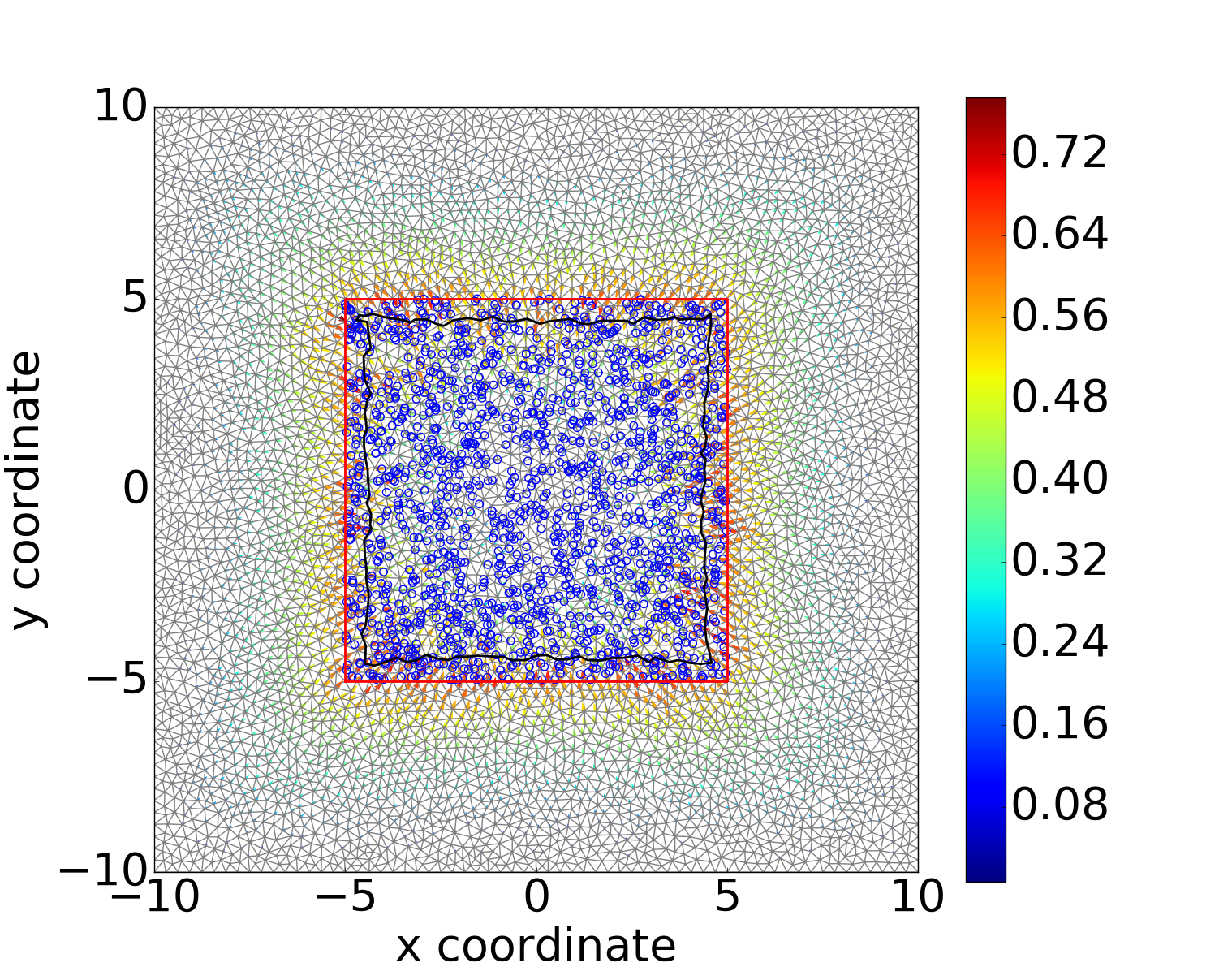}	
\caption{Identical equal-area square is used to approximate all cells. The blue circles are the cell positions, the red line and black curve present the original and deformed boundary of $\Omega_w$, respectively.}
\label{Fig_multi_eqarea_sq}
\end{figure}

\begin{figure}[htpb]
\centering
\includegraphics[width=0.85\textwidth]{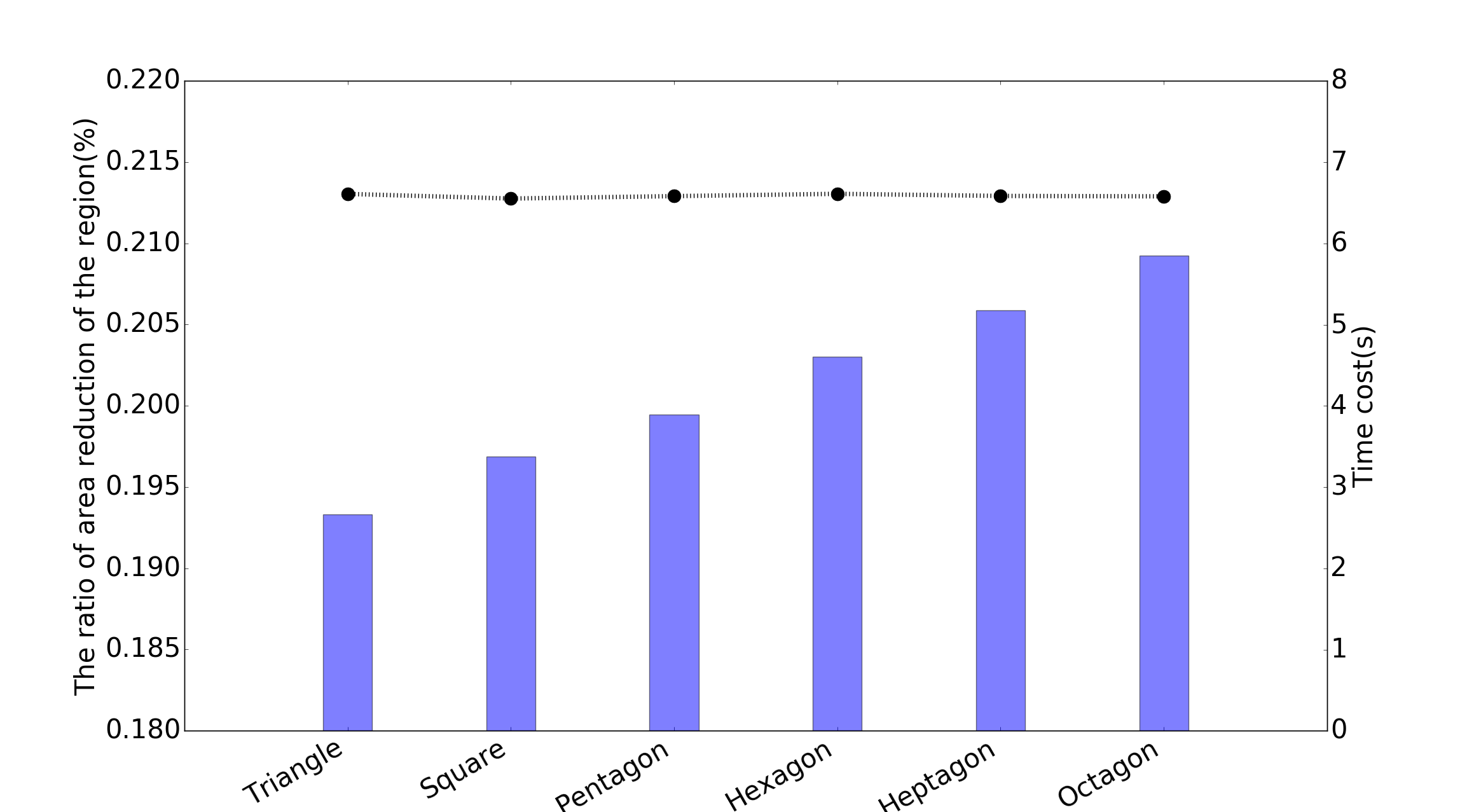}	
\caption{The blue bars indicate the computational cost; the curves display the relative reduction ratio of the subdomain area.}
\label{Fig_multi_eqarea_ratio_time}
\end{figure}

\section{Discussion and Conclusions}\label{Conclusion}
In this paper, we mainly discussed different approaches to solve linear elasticity problems with point sources forces that are exerted on cell boundaries. In order to simulate wound contraction, it is crucially important to solve the equations for balance of momentum. The body forces are determined by (myo)fibroblasts that exert forces on their immediate extracellular environment. Since we model the forces by the use of point forces which makes the solution not be in the $\boldsymbol{H}^1$ Sobolev space for dimensions exceeding one, we analysed the relation between the immersed boundary approach and the 'hole approach' and it has been computationally illustrated that the transition from the immerse boundary to the 'hole approach' has a continuous nature with respect to the elasticity in the cellular region. We proved that the finite-element approximations of the two approaches are the same if the stiffness in the cell is neglected. For large numbers of (migrating) cells, it becomes very beneficial to reduce the polygonal order of the representation of the cell boundary. The results indicate that an approximation of a cell boundary by a triangle or square is already sufficiently accurate, and the triangular representation is the least time-consuming. Furthermore, the computation of the subdomain area by the use of connecting all the boundary vertices to compute a 'polygon' area is the most efficient procedure, combined with applying shoelace method.

\section*{Acknowledgement}
\noindent
Authors acknowledge the Chinese Scholarship Council for financial support to this project.



\section*{References}
\bibliographystyle{elsarticle-num} 
\bibliography{elsarticle-template-num}


%\begin{thebibliography}{00}
%%% Text of bibliographic item
%
%\bibitem{}
%
%\end{thebibliography}
\end{document}